\documentclass[a4paper]{article}
\usepackage{geometry}                
\usepackage{mathtools}
\usepackage{amssymb}
\usepackage{amsmath}
\usepackage{amsthm,dsfont,yfonts}
\usepackage{rotating}
\usepackage{graphicx,pspicture}
\usepackage{array}
\usepackage{calc}
\usepackage{soul}
\usepackage[all]{xy}
\usepackage{amsthm}
\usepackage{float}
\usepackage{lmodern}
\usepackage{color}
\usepackage{pstricks,pst-node,pst-tree}
\usepackage{graphics,graphicx}
\usepackage[british]{babel}
\usepackage{fancyhdr}
\usepackage{textcomp}
\usepackage{enumerate}

\newtheorem{theorem}{Theorem}[section]
\newtheorem{cor}[theorem]{Corollary}
\newtheorem{lemma}[theorem]{Lemma}
\newtheorem{prop}[theorem]{Proposition}

\theoremstyle{definition}
\newtheorem{definition}[theorem]{Definition}

\theoremstyle{remark}
\newtheorem{remark}[theorem]{Remark}

\newcommand{\Aut}{\operatorname{Aut}}
\newcommand{\rst}{\operatorname{rst}}
\newcommand{\St}{\operatorname{St}}
\newcommand{\Sym}{\operatorname{Sym}}

\newcommand{\h}{\hspace{2mm}}  

\usepackage[noindentafter]{titlesec}
\title{A finitely generated branch group of exponential growth without free subgroups}
\date{}
\author{Elisabeth Fink}

\begin{document}

\selectlanguage{british}

\maketitle

\begin{abstract}
We will give an example of a branch group $G$ that has exponential growth but does not contain any non-abelian free
subgroups. This answers question 16 from \cite{Bartholdi} positively. The proof demonstrates how to construct a
non-trivial word $w_{a,b}(x,y)$ for any $a,b \in G$ such that
$w_{a,b}(a,b) = 1$.
The group $G$ is not just infinite. We prove that every
normal subgroup of $G$ is finitely generated as an abstract group and every proper quotient soluble. Further, $G$ has
infinite virtual first Betti number but is not large.\end{abstract}

\section{Introduction}

Groups acting on infinite rooted trees have provided remarkable examples in the last decades. Starting with Grigorchuk's
group in \cite{grigor_1} of intermediate growth branch groups received more and more attention. A standard
introduction to this topic is the survey \cite{Bartholdi} by Bartholdi, Grigorchuk and Sunik.
In their section on open questions the authors ask whether there exist branch groups which have exponential word growth
but do not contain any non-abelian free subgroups. We answer this question affirmatively by constructing explicit words
$w_{a,b}(x,y)$ for any $a,b \in G$ such that $w_{a,b}(a,b)=1$. It is a result by Grigorchuk and Zuk \cite{grigor_zuk} that the weakly branch Basilica group has exponential growth but no free subgroups. Sidki and Wilson constructed in \cite{wilson} branch
groups that contain free subgroups and hence have exponential growth. Nekrashevych proved in \cite{nekra} that branch
groups containing free subgroups fall into one of two cases. A paper by Brieussel \cite{brieussel_2} gives examples
of groups that have a given oscillation behaviour of intermediate growth rate. Work by Bartholdi and Erschler
\cite{erschler} provides examples of groups that have a given intermediate growth rate. In \cite{grigor_sunic}
Grigorchuk and Sunik prove that the Hanoi tower group on three pegs is amenable but that its Schreier graph has
exponential diameter growth. Recent work by Wilson \cite{wilson2} shows that if a finitely generated residually soluble
group has growth strictly less than $2^{n^{1/6}}$ then it has polynomial growth.

\medskip

The group $G$ in this paper will depend on an infinite sequence of primes. In order to establish that $G$ has
exponential growth and no free subgroups we have to make restrictions on this sequence. If we weaken those assumptions
we can prove by other means that $G$ is not large. We do not know whether these restrictions are necessary. We also do
not know whether our group $G$ is amenable. Motivated by a result
of Brieussel \cite{brieussel1}, we suspect that this could hold at least if the sequence of primes grows
slowly. Consideration of the abelianization of certain normal subgroups shows that $G$ has infinite virtual first Betti
number.

\medskip

Most of the examples studied in the literature are groups acting on regular, rooted, spherically transitive trees. In
this paper we look at finitely generated automorphism groups of an irregular rooted tree. A similar class of examples
was
first mentioned by Segal in \cite{segal_fifg}. A related construction was investigated by Woryna \cite{woryna}
and Bondarenko \cite{bondarenko} where the authors describe generating sets of infinite iterated wreath products.

\section{Rooted Trees and Automorphisms}

In this section we will recall some of the notation and definitions from \cite{Bartholdi} and \cite{segal_fifg}.

\subsection{Trees}

A \emph{tree} is a connected graph which has no non-trivial cycles. If $T$ has a distinguished \emph{root} vertex $r$
it is called a \emph{rooted tree}. The distance of a vertex $v$ from the root is given by the length of the path from
$r$ to $v$ and called the \emph{norm} of $v$. The number \[d_v = | \{e \in E(T): e=\left(v_1, v_2\right), v = v_1
\textnormal{ or } v=v_2\}|\] is called the \emph{degree} of $v \in V(T)$. The tree is called \emph{spherically
homogeneous} if vertices of the same norm have the same degree. Let $\Omega(n)$ denote the set of vertices of distance
$n$ from the root. This set is called the $n$-th level of $T$. A spherically homogeneous tree $T$ is determined by,
depending on the tree, a finite or infinite sequence $\bar{l}=\left\{l_n\right\}_{n=1}$ where $l_n+1$ is the degree of
the vertices on level $n$ for $n \geq 1$. The root has degree $l_0$. Hence each level $\Omega(n)$ has $\prod_{i=0}^{n-1}
l_i$ vertices. Let us denote this number by $m_n = |\Omega(n)|$. We denote such a tree by $T_{\bar{l}}$. A tree is
called \emph{regular} if $l_i = l_{i+1}$ for all $i \in \mathbb{N}$. Let $T[n]$ denote the finite tree where all
vertices have norm less or equal to $n$ and write $T_v$ for the subtree of $T$ with root $v$.
For all vertices $v,u \in \Omega(n)$ we have that $T_u \simeq T_v$. Denote a tree isomorphic to $T_v$ for $v \in
\Omega(n)$ by $T_n$. This will be the tree with defining sequence $\left(l_n, l_{n+1}, \dots \right)$. To each sequence
$\bar{l}$ we associate a sequence $\left\{X_n\right\}_{n \in \mathbb{N}}$ of alphabets where $X_n=\left\{v_1^{(n)},
\dots, v_{l_n}^{(n)}\right\}$ is an $l_n$-tuple so that $|X_n|=l_n$.  A path beginning at the root of length $n$ in
$T_{\bar{l}}$ is identified with the sequence ${x_1,\dots, x_i, \dots, x_n}$ where $x_i \in X_i$ and infinite paths are
identified in a natural way with infinite sequences. Vertices will be identified with finite strings in the alphabets
$X_i$. Vertices on level $n$ can be written as elements of $Y_n  = X_0 \times \dots \times X_{n-1}$. Alphabets induce
the lexicographic order on the paths of a tree and therefore the vertices.

\subsection{Automorphisms}

An \emph{automorphism} of a rooted tree $T$ is a bijection from $V(T)$ to $V(T)$ that preserves edge incidence and the
distinguished root vertex $r$. The set of all such bijections is denoted by $\Aut T$. This group induces an imprimitive
permutation on $\Omega(n)$ for each $n \geq 2$. Consider an element $g \in \Aut(T)$.  Let $y$ be a letter from $Y_n$,
hence a vertex of $T[n]$ and $z$ a vertex of $T_n$. Then $g(y)$ induces a vertex permutation $g_y$ of $Y_n$. If we
denote
the image of $z$ under $g_y$ by $g_y(z)$ then \[g(yz)= g(y)
g_y(z).\]

\medskip

With any group $G \leq \Aut T$ we associate the subgroups \[\St_G(u)=\left\{g \in G: u^g=u\right\},\] the
\emph{stabilizer} of a vertex $u$. Then the subgroup \[\St_G(n)=\bigcap_{u \in \Omega(n)} \St_G(u)\] is called the
\emph{$n$-th level stabilizer} and it fixes all vertices on the $n$-th level. Another important class of subgroups
associated with $G \leq \Aut T$ consists of the \emph{rigid vertex stabilizers} \[\rst_G(u)=\left\{g \in G: \forall v
\in
V(T)
\setminus V(T_u): v^g=v\right\}.\]
The subgroup \[\rst_G(n)= \rst_G(u_1) \times \dots \times \rst_G(u_{m_n})\] is called the \emph{$n$-th level rigid
stabilizer}. Obviously $\rst_G(n) \leq \St_G(n)$.

\begin{definition}
Let $G$ be a subgroup of $\Aut(T)$ where $T$ is as above. We say that $G$ acts on $T$ as \emph{branch
group} if it acts transitively on the vertices of each level of $T$ and $\rst_G(n)$ has finite index for all $n \in
\mathbb{N}$.
\end{definition}
The definition implies that branch groups are infinite and residually finite groups. We can specify an automorphism $g$
of $T$ that fixes all vertices of level $n$ by writing $g = \left(g_1, g_2, \dots, g_{m_n}\right)_n$ with $g_i \in
\Aut\left(T_n\right)$ where the subscript $n$ of the bracket indicates that we are on level $n$. Each automorphism can
be written as $g = \left(g_1, g_2, \dots, g_{m_n}\right)_n \cdot \alpha $
with $g_i \in \Aut\left(T_n\right)$ and $\alpha$ an element of $\Sym\left(l_{n-1}\right) \wr \dots \wr
\Sym\left(l_0\right)$. Automorphisms acting only on level $1$ by permutation are called \emph{rooted automorphisms}. We
can identify those with elements of $\Sym\left(l_0\right)$.

\section{The Construction}

In this subsection we describe the main construction of the group. The trees in this paper will have a defining
sequence $\left\{ l_i \right\}_{i \in \mathbb{N}}$ where all $l_i$ are pairwise distinct primes
greater or equal than $7$. This essentially ascending valency will prove to be the key to the exponential growth and the
non-existence of non-abelian free subgroups. The group $G$ constructed here is finitely generated, but recursively
presented. We shall prove that for every normal
subgroup $N \neq 1$, $N$ is finitely generated as an abstract group and that $G/N$ is soluble.

\subsection{The Generators}

Let $\left\{l_i\right\}_{i \in \mathbb{N}}$ be a sequence of finite cyclic groups $\{A_i\}_{i \in \mathbb{N}} $ of
pairwise coprime orders $l_i=|A_i|$ where
$l_i \geq
7$. Fix a generator $a_i$ for each $A_i$. Let us consider the rooted tree with defining sequence $\{l_i\}_{i \in
\mathbb{N}}$ and recall \[m_n = \prod_{i=0}^{n-1} l_i.\] Then each layer $n$ has $m_n$ vertices, given by the set
$\Omega(n)$. We study the group \[G=\left<a_0, b\right> \h \h\] where $a_0$ is the chosen generator of $A_0$
acting as rooted
automorphism and $b$ is recursively defined on each level $n$ by 
\[b_n=\left(b_{n+1},a_{n+1},1,\ldots,1\right)_n\] where $a_n$ is the generator of the group $A_n$. This means the
action
on the first vertex of level $1$ is given by $b_{n+1}$ and the action on the second vertex by the rooted automorphism
$a_{n+1}$. Figure \ref{figure_b} shows the action of the automorphism $b$ on the tree. The action of $b$ on all
unlabelled vertices $v$ in the figure will be given by the identity on $T_u$.

\begin{figure}[ht!]
\centering
\includegraphics[scale=0.8]{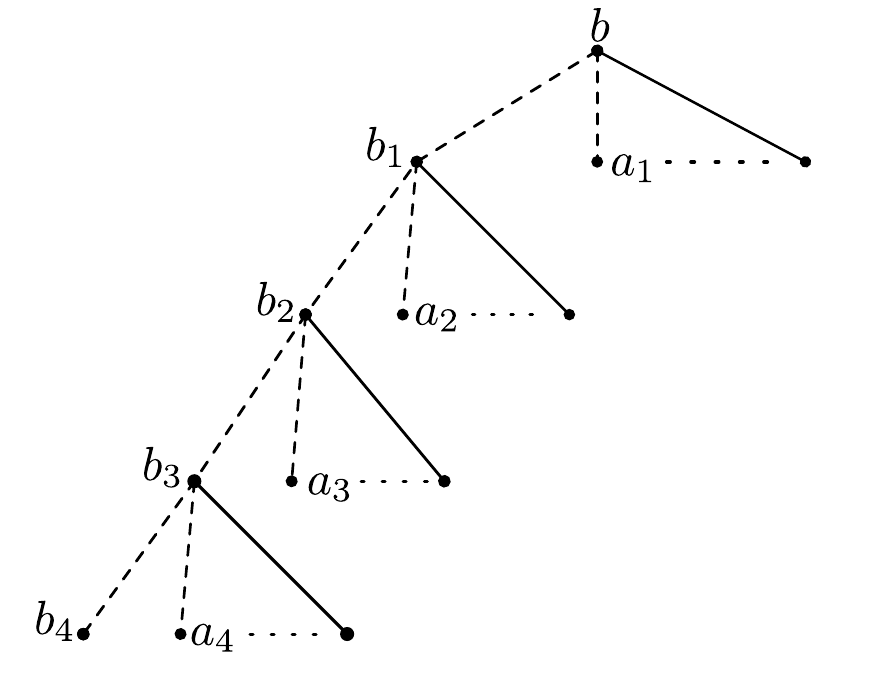}
\caption{The automorphism $b$.}\label{figure_b}
\end{figure}

\begin{prop}$G$ acts as the iterated wreath product $A_{n-1} \wr \dots \wr A_1 \wr A_0$ on the set $\Omega(n)$ of $m_n$
vertices of each level $n$.
\end{prop}

\begin{proof}
We argue by induction. The action on level $1$ is given by $A_0$. Now assume that the action of $G$ on
$\Omega(n-1)$ is given by $A_{n-2} \wr \dots \wr A_0$. The automorphism $b^{m_{n-1}}$ acts as $a_{n-1}^{m_{n-1}}$ on $v
\in \Omega(n-1)$ and trivially above level $n-1$. There exists an integer $q$ such that $a_{n-1}^{q \cdot m_{n-1}} =
a_{n-1}$ because $l_{n-1}$ and $m_{n-1}$ are coprime. Hence for all $a_{n-1}^k \in A_{n-1}$ there exists a $g=b^{q
m_{n-1}} \in G$ such that $g|_{T_{v}}=a_{n-1}^k$. This holds for any vertex of level $n-1$ by the transitivity of
$A_{n-2} \wr \dots \wr A_0$. Therefore $G$ induces the action of $A_{n-1} \wr \dots \wr A_0$ on $\Omega(n)$.
\end{proof}

\begin{cor}\label{cor_GmodStab}
$G/\St_G(n) = A_{n-1} \wr \dots \wr A_0$.
\end{cor}

We denote conjugation by $x^y = y^{-1} x y$ and commutators by $[x,y]=x^{-1}y^{-1}xy$. Define the following
automorphisms and groups: \[b(i) = b^{a^{i-1}} \h \textnormal{
for } i = 1, \dots, l_0\] and
similarly \[b_n(i) = b_n^{a_n^{i-1}} \h \textnormal{ for } i = 1, \dots, l_n.\]

Also define
\[B_n=\left<b_n(1), \dots, b_n\left(l_{n-1}\right)\right>\] for $n \geq 0$ and similarly to $G$ the groups
\[G_n=\left<a_n, b_n\right>\] for $n \geq 1$. Write $G_0 = G$, $B=B_0$ and $A=A_0$.

\begin{prop}\label{prop_spinal} With the above definitions we get the following statements:
\begin{enumerate}[(a)]
 \item $G=B \rtimes A$ and so $G'=B'\cdot \left<[B,A]\right>$.
 \item $\St_G(1) \leq G_1 \times \dots \times G_1$.
 \item $B=\St_G(1)$.
\end{enumerate}
\end{prop}

\begin{proof}
\begin{enumerate}[(a)]
\item Clearly $B \cap A = 1$ and $B \lhd G$.
\item $\St_G(1)$ is generated by $a$-conjugates of $b_0=\left(b_1, a_1, 1, \dots, 1\right)$. But $b_1$
and $a_1$ are in $G_1$, hence $\St_G(1)\leq G_1 \times \dots \times G_1$.
\item We see that $B \leq \St_G(1)$. For the other inclusion we use $G=B \cdot A$ and 
the modular law with $B \leq \St_G(1)$. We get $\St_G(1) = B (A \cap \St_G(1)) = B$ because $A \cap \St_G(1) = 1$.
\end{enumerate}
\end{proof}

Write $\Gamma'$ for the derived subgroup $\left[\Gamma,\Gamma\right]$ of a group $\Gamma$ and by $\Gamma^{(n)}$ for $n
\geq 1$ the $n$-th derived subgroup $\Gamma^{(n)} = \left[\Gamma^{(n-1)}, \Gamma^{(n-1)}\right]$ where
$\Gamma^{(0)}=\Gamma$.

\begin{lemma}\label{lemma_comm_inrst}
$B'B^{l_1} \leq \rst_G(1)$.
\end{lemma}

\begin{proof}
We first prove $B' \leq \rst_G(1)$ and claim that

\begin{equation}\label{eq_comm} 
[b(i), b(j)]=
\begin{cases} 
\left(1,\dots,1,\left[a_1,b_1\right],1,\dots,1 \right)_1 & \textnormal{ if } \h j=i+1 \mod l_0 \h \\
\left(1,\dots,1,\left[b_1,a_1\right],1,\dots,1 \right)_1 & \textnormal{ if } \h i=j+1 \mod l_0 \h \\
1 & \textnormal{ otherwise.} 
\end{cases} 
\end{equation}

We look at the action of $\left[b(i),b(j)\right] = b(i)^{-1}b(j)^{-1}b(i)b(j)$ on the first layer for the first and
third case. The second one follows similarly. Denote by underbracing the positions of the respective elements.

\begin{itemize}

\item $j = i+1$:

\[b(i)^{-1}b(j)^{-1}b(i)b(j) = (1, \dots, 1, \underbrace{b_1^{-1}b_1}_{i}, \underbrace{a_1^{-1}
b_1^{-1}a_1 b_1}_{j=i+1}, \underbrace{a_1^{-1} a_1}_{j+1}, 1 \dots, 1)_1\]
\[= \left(1, \dots, 1, \left[a_1, b_1\right], 1, \dots, 1 \right)_1. \]

\item $|i-j| > 1$: \[b(i)^{-1}b(j)^{-1}b(i)b(j) = (1, \dots, 1, \underbrace{b_1^{-1}b_1}_{i}, \underbrace{a_1^{-1}
a_1}_{i+1}, 1, \dots, 1, \underbrace{b_1^{-1}b_1}_{j}, \underbrace{a_1^{-1} a_1}_{j+1}, 1 \dots, 1)_1=1.\]
 
\end{itemize}

It remains to show $B^{l_1} \leq \rst_G(1)$.
\[ b(k)^{l_1}=(1, \dots, 1, \underbrace{b_1^{l_1}}_{k}, a_1^{l_1}, 1, \dots, 1)_1
=(1,\dots,1,\underbrace{b_1^{l_1}}_{k},1, \dots,1)_1 \in \rst_G(1) \h\textnormal{ for } i=1,\dots,l_0.\]
\end{proof}

\subsection{Introducing $N$}

In this subsection we define a normal subgroup $N$ that will be proved to be equal to the derived group of $G$.
However, this explicit construction and the explicit finite set of generators that we will obtain will be very
useful.

\medskip

Let $F_{l_0} = \left<x_1, \dots, x_{l_0}\right>$ be the free group on $l_0$ generators. The map 
\begin{equation}\label{eq_f}
f: 
\begin{cases} 
F_{l_0} \longrightarrow \mathbb{Z}\\
x_i \mapsto 1
\end{cases}
\end{equation} 

is surjective. Its kernel $K(x_1, \dots, x_{l_0})=\ker(f)$ consists of all words in the generators where the sum over
all
exponents is $0$.

\begin{lemma}\label{lemma_N}
$K\left(x_1, \dots, x_{l_0}\right)=\left<x_i^{-1}x_j | i,j= 1,\dots, l_0 \right>^F$.
\end{lemma}

\begin{proof}
Define $X=\left<x_i^{-1}x_j | i,j= 1,\dots, l_0 \right>^F$. We first show $F' \leq X$. We can write \[x_i^{-1} x_j^{-1}
x_i x_j = \left(x_j^{-1}x_i\right)^{x_i} \cdot x_i^{-1}
x_j\] which proves the claim. Clearly $X \leq K$. We observe that $K/F'=X/F'$ which yields that $K=X$.

\end{proof}

Define 
\[N_n=K\left(b_n(1), \dots, b_n\left(l_n\right)\right)\] for $n \geq 0$ and write $N = N_0$ for the rest of this
paper. The following lemma follows straight from the definition.

\begin{lemma}
$N_n \leq B_n$ for $n \geq 0$.
\end{lemma}

\begin{lemma}\label{lemma_Nfg} The subgroup $N$ is finitely generated by $\left\{b(2)^{-1}b(1), b(3)^{-1}b(2), \dots,
b(1)^{-1} b\left(l_0\right)\right\}$.
\end{lemma}

The essential property used in this proof is that each generator of $B$ commutes with most of the others.
More precisely we have the identities $\left[b(i), b(k)\right]=1$ if $|i-k| \neq 1 \mod l_0$.

\begin{proof} We need here that $l_0 \geq 6$ and set $D=\left<b(2)^{-1}b(1),b(3)^{-1}b(2),
\dots, b(1)^{-1}b\left(l_0\right)\right>$. We show that \[\left(b(2)^{-1}b(1)\right)^{b(k)}
\in D.\]

We first show that all elements of the form $b(j)^{-1}b(i)$ and $b(j)b(i)^{-1}$ for any $i,j = 1, \dots, l_0$ are in
$D$.
The first one is easy to see by taking products of consecutive elements. For $b(j)b(i)^{-1}$ we build
$b(i)b(i-1)^{-1}$ first:
\[b(i)b(i-1)^{-1} = b(i+2)^{-1}b(i+2)\cdot  b(i) b(i-1)^{-1} = b(i+2)^{-1} b(i) \cdot b(i-1)^{-1} b(i+2).\]
This is a product of two elements which are already in $D$ because we have $[b(i-1),b(i+2)]=[b(i),b(i+2)]=1$. We only
need to prove
closure under conjugation by $B$. It remains to look at $k=i-1, k=i$ and $k=i+1$. If without loss of generality $k=i$ or
$k=i+1$, we have

\[\left(b(i)^{-1}b(i-1)\right)^{b(k)} = b(k)^{-1}b(i)^{-1}b(i-1)b(k)\] 
\[=b(k)^{-1}b(k+2)\cdot b(k+2)b(i)^{-1} \cdot b(i-1)b(k+2)^{-1} \cdot b(k+2)^{-1}b(k),\]
because $b(k+2)$ commutes with all other factors in this expression if $k+2 \neq i-1$. The latter is a product of
four elements in
$D$. The cases $k=i-1$ and $k=i-2$ can be dealt with in the same way. Therefore $D^b \leq D$ for all $b \in B$ and so
$D^B=D$. $N=D$ and so $N$  is finitely generated because $D$ obviously is.
\end{proof}

\begin{prop}\label{prop_gnDash}
$G_n'=N_n$ for $n \geq 0$.
\end{prop}

\begin{proof}
$N$ is the kernel of a map whose image is abelian hence $G' \leq N$. Looking at the generators of $N$ we see that
$N/G'=1$ and hence the groups are equal.
\end{proof}

\begin{lemma}\label{lemma_Bdash}
$B'= N_1 \times \dots \times N_1$ and so $B'\leq B_1 \times \dots \times B_1$.
\end{lemma}

\begin{proof}
We have $B=\St_G(1)\leq G_1 \times \dots \times G_1$ and hence $B'\leq G_1'\times \dots \times G_1'=N_1 \times \dots
\times N_1$ by Corollary \ref{prop_gnDash}. We now prove $N_1' \times \dots \times N_1' \leq B'$. The group $N_1$ is
generated by elements of the form $b_1(j+1)^{-1}b_1(j) = \left[b(1), b(2)\right]^{b(1)^{j-1}}$ and hence in $B'$.
\end{proof}

\begin{cor}\label{cor_BnDerived}
$B_{n-1}'=N_n \times \dots \times N_n \leq B_n \times \dots \times B_n$ for $n \geq 1$.
\end{cor}

\begin{lemma}\label{lemma_generatingN}\label{lemma_N_derived}\label{lemma_prod}\label{lemma_Gder}
We have the following identities for the subgroups defined above for $n \geq 0$:
\begin{enumerate}[(a)]
 \item  \label{part_1} $N'=B'$.
 \item \label{part_2} $N_n'= N_{n+1} \times \dots \times N_{n+1}$ with $l_n$ factors in the direct product.
 \item \label{part_3} $G^{(n+1)} = G_n'\times \dots \times G_n'$ with $m_n$ factors in the
direct product.
 \item  \label{part_4} $G^{(n+1)} \subseteq \rst_G(n)$. 
\end{enumerate}
\end{lemma}

\begin{proof}
\begin{enumerate}[(a)]
\item  Elementary commutator manipulation shows that \[\left[b(2),b(1)\right] = \left[b(4)^{-1}b(2),
b(2)^{-1}b(1)\right].\]
This implies $B' \leq N'$. The other inclusion follows straight from $N \leq B$.
\item By Corollary \ref{cor_BnDerived} and we have $N_n'=B_n'=N_{n+1} \times \dots \times N_{n+1}$.
\item We start with \[G^{(n+1)}=\left(G'\right)^{(n)} = N^{(n)}=\left(N'\right)^{(n-1)}=(\underbrace{N_1 \times \dots
\times N_1}_{l_o \textnormal{ times}})^{(n-1)} = N_1^{(n-1)}\times \dots \times N_1^{(n-1)}\]

and apply (\ref{part_2}) iteratively together with Proposition \ref{prop_gnDash} and get

\[\underbrace{N_n \times \dots \times N_n}_{m_n \textnormal{ times}} = G'_n \times \dots \times G'_n.\]
\item The proof of (\ref{part_3}) implies $G^{(n+1)}=N_n \times \dots \times N_n \leq (G \cap G_n) \times \dots \times
(G \cap
G_n)=\rst_G(n)$.
\end{enumerate}
\end{proof}

\begin{cor}
$B_n''=B_{n+1}' \times \dots \times B_{n+1}'$ and $B^{(n)}=B_{n-1}'\times \dots \times B_{n-1}'$ for $n \geq 0$.
\end{cor}

\begin{lemma}
$\St_G(n) = G \cap \left(G_n \times \dots \times G_n\right)$ for $n \geq 0$.
\end{lemma}

\begin{proof}
It is obvious that $G \cap \left(G_n \times \dots \times G_n\right)$ is contained in $\St_G(n)$. The other
inclusion is given by Proposition \ref{prop_spinal} for $n=1$ and follows iteratively from $\St_G(n+1) \leq
\St_{\St_G(n)}(1)$.
\end{proof}

\begin{lemma}\label{eq_b_powers}
$b^{m_{n+1}} = \left(b_n^{m_{n+1}},1,\dots,1\right)_n=\left(b_{n-1}^{m_{n+1}},1,\dots,1\right)_{n-1} \in G$ for $n \geq
0$.
\end{lemma}

\begin{proof}
Every $a_n$ has order $l_n$. Hence $\left(b_{n}, a_n, 1, \dots, 1, \right)_n^{l_0 \dots l_n} =
\left(b_n^{l_0, \dots, l_n}, 1, \dots,
1\right)_n$.
\end{proof}

\begin{lemma}\label{lemma_BnInG}\label{lemma_Bndash_In_rst_G}
The following statements hold for $n \geq 0$:
\begin{enumerate}[(a)]
 \item $B_n' \cdot B_n^{m_{n+1}} \leq G$ where $B_n^{m_{n+1}} = \left<b_n(i)^{m_{n+1}}\right>$ for $n \in
\mathbb{N}$. 
 \item $B_{n-1}' B_{n-1}^{m_n} \times \dots \times B_{n-1}' B_{n-1}^{m_n} \leq \rst_G(n)$.
\end{enumerate}
\end{lemma}

\begin{proof}
Lemma \ref{lemma_prod} implies $B_n'\times \dots \times B_n' = N_n' \times \dots
\times N_n' \leq N_n \times \dots \times N_n = G^{(n+1)}$ and together with Lemma \ref{eq_b_powers} this gives $B_n'
\cdot
B_n^{m_{n+1}} \leq G$ which proves both parts.
\end{proof}

\begin{lemma}
$\left(G_{n+1} \times \dots \times G_{n+1} \cap G\right)' = G_{n+1}' \times \dots \times G_{n+1}'$ for $n \geq 0$.
\end{lemma}

\begin{proof}
We have $G_n' \times \dots \times G_n' \leq \left(G_{n+1} \times \dots \times G_{n+1}\right) \cap G$ and we get 
\[G_{n+1}' \times \dots \times G_{n+1}' = G_n'' \times \dots \times G_n'' \leq \left( \left( G_{n+1} \times \dots \times
G_{n+1}\right) \cap G \right)'\]
\[\leq \left(G_{n+1}' \times \dots \times G_{n+1}'\right) \cap G' = G^{(n+2)} \cap G' = G^{(n+2)} = G_{n+1}' \times \dots \times G_{n+1}'. \]
\end{proof}

\begin{lemma}\label{lemma_rst_inStG}
The following statements hold:
\begin{enumerate}[(a)]
 \item $\rst_G(1) = B' \cdot B^{l_1}$.
 \item $\rst_G(n) \leq \prod_{i=1}^{m_{n-1}} B_{n-1}' \cdot B^{m_{n+1}}$  for $n \geq 1$.
 \item \label{lemma_rst_inStGPart3} $\rst_G(n) \leq \St_G(n+1)$ for $n \geq 1$.
\end{enumerate}
\end{lemma}

\begin{proof}
We first see that $\rst_G(1)= B' \cdot B^{l_1}$ because of Lemma \ref{lemma_Bndash_In_rst_G} and $\rst_G(1) \leq
B=\St_G(1)$. Hence $\rst_G(n) \leq \prod_{i=1}^{m_{n-1}} \rst_{G_{n-1}}(1) = \prod B_{n-1}' \cdot B^{m_{n+1}}$
which fixes layer $n+1$.
\end{proof}

\begin{prop}\label{cor_rstdash_fg}
$\rst_G(n)'=G^{(n+2)}$ for $n \geq 1$, in particular $\rst_G(n)'$ is finitely generated.
\end{prop}

\begin{proof}
Lemma \ref{lemma_rst_inStG} states $\rst_G(1) = B' \cdot B^{l_1}$ and therefore
$\rst_G(1)' = B'' \cdot \left[B', B^{l_1}\right]
\left(B^{l_1}\right)'$. For the first group we have $B'' = B_1' \times \dots \times B_1'$ and for the last one we see
that $B^{l_1} \leq B_1 \times \dots \times B_1$. It therefore remains to observe that $\left[B', B^{l_1} \right] \leq
\prod B_1'$ which follows from $B' \leq B_1$ and $B^{l_1} \leq B_1$. This implies $\rst_G(1)' = B_1' \times \dots \times
B_1' = N_2 \times \dots \times N_2$ by Corollary \ref{cor_BnDerived} which is finitely generated. It is now left to show
that this implies $\rst_G(n)'$ is finitely generated for all $n \in \mathbb{N}$. By Lemma
\ref{lemma_rst_inStG}(\ref{lemma_rst_inStGPart3}) we have the following inclusions:

\[ \rst_G(n)' \leq \left(G_{n+1} \times \dots \times G_{n+1} \cap G\right)'\] \[= G_{n+1}'\times \dots \times
G_{n+1}' = \left(G_n'
\times \dots \times G_n'\right)' \leq \rst_G(n)'\] because $G_{n+1}'\times \dots \times G_{n+1}' = G^{(n+2)}
\leq G'$.
So by this we have \[N_{n+1} \times \dots \times N_{n+1} = G_{n+1}' \times \dots \times G_{n+1}' = \left(G_n' \times
\dots \times G_n'\right)' = \rst_G(n)'\] which is therefore finitely generated by Lemma \ref{lemma_Nfg}.\\
\end{proof}

\begin{theorem}\label{thm_quotFinite}
The group $G$ is a branch group. Further the quotient $\frac{\St_G(n)}{\rst_G(n)}$ for $n \geq 1$ is abelian and has
exponent dividing $l_1 l_2 \dots l_{n-1}  l_n$.
\end{theorem}

\begin{proof}
In the case $n=1$ we have $\St_G(1) = B$. We have $\St_G(n) \leq B_{n-1} \times \dots \times B_{n-1}$ for $n > 1$ and so
\[\St_G(n)^{l_1  \dots  l_n} \leq \left(B_{n-1} \times \dots \times
B_{n-1}\right)^{l_1  \dots  l_n} = \left(B_{n} \times \dots \times B_{n}\right)^{l_1  \dots  l_n} \leq \rst_G(n)\] by
Lemma \ref{eq_b_powers}. Now Lemma \ref{lemma_prod} implies \[\St_G(n)'=G' \cap (G_n'\times \dots
\times G_n') = G' \cap G^{(n+1)} \leq \rst_G(n).\] The quotient $\frac{\St_G(n)}{\rst_G(n)}$ is therefore abelian and
has
exponent dividing $l_1 l_2 \dots  l_{n-1}  l_n$. The $n$-th level stabilizers $\St_G(n)$ always have finite index, hence
$\rst_G(n)$ is of finite index in $G$.
\end{proof}

\begin{lemma}\label{lemma_cyclic}
$\frac{B_n}{N_n} \simeq \mathbb{Z}$ for $n \geq 0$.
\end{lemma}

\begin{proof}
Let $F_{l_0}=\left<x_1, \dots, x_{l_0}\right>$ be the free group on $l_0$ generators and $\pi$ the
natural projection \[\pi: \begin{cases}F_{l_0} & \longrightarrow B \\ x_i & \longmapsto b(i) \in
B\end{cases}.\] 

The map from equation (\ref{eq_f}) together with the natural injection \[\iota: \begin{cases}N(x_1,\dots, x_{l_0}) &
\hookrightarrow
F_{l_0} \\ x_i & \mapsto x_i \h\h \textnormal{ for all } \h i=1,\dots,l_0\end{cases}\]
gives the following sequence:

\begin{center}
\begin{tabular}{ccccccccc}
  $1$ & $\hookrightarrow$ & $N(x_1,\dots,x_{l_0})$ & $\hookrightarrow$ & $F_{l_0}$ &
$\twoheadrightarrow$ & $\mathbb{Z}$ & $\rightarrow$ & 0\\
  & & $\downarrow \pi$ \\
  & & $N \pi \leq B$
\end{tabular}
\end{center}

We see that $F_{l_0}/N \simeq \mathbb{Z}$
and hence its image $B/N$ under $\pi$ must be an infinite cyclic
group.
\end{proof}

\subsection{Finite Generation of Normal Subgroups}

We quote a theorem by Grigorchuk \cite{grigor_bg}.

\begin{theorem}\label{lemma_grigor}
Let $\Gamma \leq \Aut(T)$ be a spherically transitive subgroup of the full automorphism group on $T$. If $1 \neq N \lhd
\Gamma$, then there exists an $n$ such that $\rst_\Gamma(n)' \leq N$.
\end{theorem}

\begin{prop}\label{prop_quotientSoluble}
Every proper quotient of $G$ is soluble.
\end{prop}

\begin{proof}
This follows straight from Theorem \ref{lemma_grigor} and Lemma \ref{lemma_Gder}.
\end{proof}

\begin{theorem}
In the group $G$ defined above every normal subgroup is finitely generated.
\end{theorem}

\begin{proof}
By Lemma \ref{lemma_grigor} every normal subgroup $K \lhd G$ contains some $\rst_G'(n)$. Corollary \ref{cor_rstdash_fg}
states that $\rst_G(n)'$ is finitely generated. So it suffices to show that $K/\rst_G(n)'$ is finitely generated. The
group $K/\rst_G(n)'$ is a finite extension of the finitely generated abelian group $\left(K\cap
\rst_G(n)\right)/\rst_G'(n)$.
\end{proof}

\subsection{Congruence Subgroup Property}

We recall that a branch group $\Gamma$ has the \emph{congruence subgroup property} if for every subgroup $H \leq \Gamma$
of finite index in $\Gamma$ there exists an $n$ such that $\St_\Gamma(n) \leq H$.

\begin{theorem}
$G$ does not have the congruence subgroup property.
\end{theorem}

\begin{proof}
The quotient \[\frac{\rst_G(n)}{\rst_G(n)' \rst_G(n)^p}\] is an elementary abelian $p$-section for every prime $p$. By
taking
$n$ large enough we can find $p$-sections of arbitrarily large rank in $G$. Because $G$ is a branch group $\rst_G(n)$
has
finite index. On the other hand any congruence quotient $G/H$ is a quotient of $A_{k-1} \wr \dots
\wr A$ for some $k \in \mathbb{N}$. Hence its $p$-rank is finite and determined by the sequence of primes we chose.
\end{proof}

This implies that the profinite completion maps onto the congruence completion with non-trivial congruence kernel.

\begin{theorem}
The rank of $\frac{\St_G(n+1)}{\rst_G(n)}$ is less than or equal to $m_{n+1}= \prod_{i=0}^{n} l_i$ for $n \geq 0$.
\end{theorem}

\begin{proof}
The inclusions $\St_G(n+1) \leq \prod_{i=1}^{m_n} B_n$ and $N_n\times \dots \times
N_n=G^{(n+1)} \leq \rst_G(n)$ give that the quotient $\frac{\St_G(n+1)}{\rst_G(n)}$ is a section of $\frac{B_n
\times \dots \times B_n}{N_n \times \dots \times N_n}$. Hence the first quotient has rank less than or equal to
$m_{n+1}$
by Lemma \ref{lemma_cyclic}.
\end{proof}

\section{Abelianization}

This section is devoted to computing the abelianization $G^{ab} = G/G'$ of $G$ where $G'$ is the derived group. This
will allow us to determine the abelianizations of the $n$-th level rigid stabilizers $\rst_G(n)$. Considering those we
show
that the virtual first Betti number of $G$ is infinite.

\subsection{Abelianization of $G$}

The abelianization of $G$ as a $2$-generator group must be an image of the free abelian
group $F_2^{ab}=\left<x_1,x_2\right>$ on two generators, in particular an image of $F_2^{ab}=C_\infty \times
C_\infty$.

\begin{theorem}\label{thm_ab_cyclic}
$G^{ab}=C_{l_0} \times C_\infty$.
\end{theorem}

\begin{proof}
The abelianization can be presented as $G^{ab}=\left<a,b | a^{e_i}b^{d_i}=1\right>$ for possibly infinitely
many pairs of exponents $e_i, d_i
\in \mathbb{Z}$. By
construction the order of $a$ is $o(a)=l_0$. We
now show that the image of $b$ has infinite order in the abelianization. Corollary
\ref{cor_GmodStab} describes the quotients \[\frac{G}{\St_G(n)} =
A_{l_{n-1}} \wr \cdots \wr A_{l_0} =: W(n).\] Consider the natural projections \[\varphi: G \twoheadrightarrow
\frac{G}{G'} \twoheadrightarrow \frac{W(n)}{W'(n)} = A_{l_{n-1}} \times \dots \times A_{l_0}.\]
The image of $b$ under the composite of these has order $o(\varphi(b)) = \prod_{i=0}^{n-1} l_i$. This order tends to
infinity with $n$ and must therefore be infinite in $G^{ab}$.
\end{proof}

\begin{cor}
$G_n^{ab} = C_{l_n} \times C_\infty$ for $n \geq 1$.
\end{cor}

\subsection{Abelianization of Subgroups}

In this subsection we determine the abelianization of the subgroups $B$ and $\rst_G(n)$. This will yield that $G$ is not
just infinite.

\begin{prop}\label{prop_Bab}
$B^{ab}\simeq \prod_{i=1}^{l_0} \mathbb{Z}$.
\end{prop}

\begin{proof}
The elements $b(i)^{l_1}$ are all in $B$. The image of each $b(i)$ in $G^{ab}$ has infinite order by the
proof of Theorem \ref{thm_ab_cyclic}. The subgroup $H=\left< b(1)^{l_1}, \dots, b\left(l_0\right)^{l_1}\right> \leq B$
is
therefore free abelian of rank $l_0$, hence $H \cap B'=1$. We get that \[H \simeq \frac{H}{H \cap B'} \simeq \frac{H
B'}{B'} \leq \frac{B}{B'}\] and hence $B^{ab}$ has rank at least $l_0$. But $B$ is generated by $l_0$ elements and so
$B^{ab} \simeq \prod_{i=1}^{l_0} \mathbb{Z}$.

\end{proof}

\begin{cor}\label{cor_BnAB}
$B_n^{ab} \simeq \prod_{i=1}^{l_n} \mathbb{Z}$ for $n \geq 1$.
\end{cor}

\begin{theorem}\label{thm_rstAB}
$\rst_G(n)^{ab}=\prod_{i=1}^{m_{n+1}} \mathbb{Z}$ for $n \geq 1$.
\end{theorem}

\begin{proof}
Proposition \ref{cor_rstdash_fg} gives the equality $\rst_G(n)' = G_{n+1}' \times \dots \times G_{n+1}'=N_{n+1} \times
\dots \times N_{i+1}$ and hence

\[\frac{\rst_G(n)}{\rst_G(n)'} = \frac{\rst_G(n)}{N_{n+1} \times \dots \times N_{n+1}} \leq \frac{\prod B_{n}}{\prod
B_n'} \simeq \prod_{i=1}^{m_{n} } \prod_{j=1}^{l_n} \mathbb{Z}.\]
It remains to prove that we have full rank $m_{n+1}= m_n \cdot l_n$. We observe that the elements $b_n(i)^{m_n}$ for
$i=1, \dots,
l_n$ all lie in $\rst_G(n)$ and have disjoint support. The subgroup\[\prod_{i=1}^{m_{n}} \left<b(1)^{m_{n+1}}, \dots,
b\left(l_n\right)^{m_{n+1}} \right>\] of $\rst_G(n)$ therefore maps onto a rank $m_{n+1}$ subgroup of $\rst_G(n)^{ab}$
which
proves the claim.
\end{proof}

\begin{cor}
$G$ is not just infinite.
\end{cor}

\begin{proof}
Theorem 4 in \cite{grigor_bg} states that $G$ is just infinite if and only if $\rst_G(n)^{ab}$ is finite for each $n$.
This together with Theorem \ref{thm_rstAB} proves the claim.
\end{proof}


We recall that the first Betti number $b_1(\Gamma)$ of a group $\Gamma$ is the dimension of $H_1(\Gamma;\mathbb{Z})
\otimes \mathbb{Q}$. This is the rank of $\Gamma^{ab}$. The virtual first Betti number of a group $\Gamma$ is defined
\cite{lackenby} to be 

\[vb_1(G) = \sup\{b_1(H): |G/H| < \infty \}. \]

\begin{cor}\label{cor_bettiNumber}
$vb_1(G)$ is infinite.
\end{cor}

\begin{proof}
Theorem \ref{thm_rstAB} states that the rank of $\rst_G(n)^{ab}$ is $m_{n+1}$ for all $n$.
\end{proof}

\section{Growth}

Denote for any finitely generated group $\Gamma = \left<X\right>$ for any element $\alpha=\prod_{i=1}^{m_\alpha}
x_{j_i}^{\pm 1}$, $x_{j_i} \in X$ in $\Gamma$ by \[|\alpha|=\min\left\{m_\alpha: \alpha=\prod_{i=1}^{m_\alpha}
x_{j_i}^{\pm 1}, x_{j_i} \in X\right\}\] the word length of $\alpha$ in the
generators $X$ of $\Gamma$. Write $\gamma_\Gamma(n) = |\left\{\alpha \in \Gamma: \left|\alpha| \leq n\right\}\right|$
for the growth function of $\Gamma$.

\begin{prop}\label{prop_nonpoly}
The group $G$ does not have polynomial growth.
\end{prop}

\begin{proof}
The free abelian group $F_n^{ab}$ of rank $n$ embeds into $G$ for all $n \in \mathbb{N}$ as the proof of Theorem
\ref{thm_rstAB} shows.
\end{proof}

\begin{lemma}\label{lemma_G_1inG}
Let $v_{3k+2}$ be the $(3k+2)$nd vertex on level $1$ for $k=0,\dots,\lfloor\frac{l_0}{3}\rfloor$. Then the action of $G$ on it is given by $G|_{v_{3k+2}} = G_1$.
\end{lemma}

\begin{proof}
The action $G|_{T_{v_2}}$ is given by $b(1)=(b_1,a_1,1,\dots,1)_1$ and $b(2)=a^{-1}ba = (1,b_1,a_1,\dots,1)$ and similarly for $v_{3k+2}$.
\end{proof}

\begin{lemma}\label{lemma_recursion}
$\gamma_G\left(m_i^2 n\right) \geq \gamma_{G_i}(n)^{\left\lfloor\frac{m_i}{3^i}\right\rfloor}$.
\end{lemma}

\begin{proof}
We need a word of length at most $3$ to get the generators of $G_1$ on $v_2$. Every word on $v_{3k+2}$ is given by one on $v_2$ conjugated by at most $a^{\pm \left\lfloor \frac{l_0}{2}\right\rfloor}$. Hence we get a recursion
\[\gamma_G\left( \left\lfloor \frac{l_0}{3}\right\rfloor(3n+l_0)\right)\geq \gamma_{G_1}(n)^{\left\lfloor \frac{l_0}{3}\right\rfloor}.\]
We can estimate this expression by $\frac{l_0}{3}(3n+l_0) \leq l_0 n + l_0^2$. Iterating this we get \[\left(l_0n + l_0^2\right)l_1 + l_1^2=l_0l_1n + l_0^2l_1 + l_1^2\] and so for the $i$-th step $m_i\left(n+\sum_{j=0}^{i-2}l_j\right)+l_{i-1}^2 \leq m_i\left(n+\sum_{j=0}^{i-1} l_j\right) \leq m_i^2n$. Hence we get
\[\gamma_G\left(m_i^2n\right) \geq \gamma_{G_i}(n)^{\left\lfloor \frac{m_i}{3^i}\right \rfloor}.\]
\end{proof}

\begin{prop}\label{prop_WordsLength_l_i_inG_i}
$\gamma_{B_i}\left(2l_i\right) \geq 2^{l_i-1}  \cdot \left(l_i-1\right)^{\frac{l_i}{2}-2}$.
\end{prop}

\begin{proof}
We place $l_i$ conjugates of $b_i$ on $l_i$ places. In particular we build words of the form \[x=a_i^{q_0} \cdot
\prod_{j=1}^{r} b_i^{s_j} a_i^{q_j}\] with $q_j \in \mathbb{N}$ for $j\in \{1, \dots,
r\}$, $q_0 \in \mathbb{N} \cup \{0\}$, $s_j \in \mathbb{Z}$ and the restriction that $\sum_{k=0}^{r} q_k \leq l_i$. We
need at most $l_i$ times the letter $a_i$ and hence at most $2 \cdot l_i$ letters for any word as the above. This is
equivalent to a weak composition of $l_i$ into $l_i$ parts. By \cite{knuth} this number is given as
$\binom{2l_i-1}{l_i-1}$ or 

\[\frac{\left(2l_i-1\right)!}{\left(l_i-1\right)!l_i!} = \frac{\left(2l_i-1\right)\cdots \left(2l_i-2\right) \cdots \left(l_i+1\right)}{\left(l_i-1\right)!}.\] Now we use that for each factor $\left(l_i-k\right)$ in the denominator for $1 \leq k \leq \left\lfloor \frac{l_i}{2}\right\rfloor$ we have exactly one factor $\left(2l_i - 2k\right)$ in the nominator. Hence the last expression simplifies to

\[\frac{1}{\left(\left\lfloor\frac{l_i}{2}\right\rfloor\right)!}2^{l_i-1} \left(2 l_i-1\right) \left(2 l_i-3\right)
\cdots \left(l_i + 1\right) \geq \frac{1}{\left(\left\lfloor\frac{l_i}{2}\right\rfloor\right)!} 2^{l_i-1} \left(2l_i -
2\right) \cdots l_i\]
\[= \frac{1}{\left(\left\lfloor\frac{l_i}{2}\right\rfloor\right)!} \cdot 2^{l_i-1} \cdot 2^{l_i-2} \cdot
\left(l_i-1\right)! \geq \frac{1}{e} \left(\frac{\left\lfloor\frac{l_i}{2}\right\rfloor
+1}{e}\right)^{-\left\lfloor\frac{l_i}{2}\right\rfloor-1}\cdot 2^{2l_i-3} \cdot e
\left(\frac{l_i-1}{e}\right)^{l_i-1} \] 
\[\geq 2 \left(\frac{4}{e^{\frac{1}{2}}}\right)^{l_i-2} \cdot \left(l_i-1\right)^{\frac{l_i}{2}-2}\geq 2^{l_i-1}  \cdot \left(l_i-1\right)^{\frac{l_i}{2}-2}.\]
\end{proof}

\begin{theorem}
If the sequence $\left\{l_i\right\}$ satisfies $\log\left(l_i-1\right) \geq 5 \cdot \left(\frac{47}{5}\right)^i
\prod_{j=0}^{i-1} l_j$ for all $i$ then the group $G$ has exponential growth rate.
\end{theorem}

\begin{proof}
We want an estimate for $\gamma_G(r)$. For this we choose $i$ such that $r \geq m_i^2 \cdot 2 l_i$. Lemma
\ref{lemma_recursion} and Proposition \ref{prop_WordsLength_l_i_inG_i} then give $\gamma_G(r) \geq 2^{l_i-1}  \cdot
\left(\left(l_i-1\right)^{\left\lfloor\frac{l_i}{2}\right\rfloor-2}\right)^{\left\lfloor\frac{m_i}{3^i}\right\rfloor}$.
We compare this expression to $e^{2 m_i^2 l_i}$ and assume $l_i \geq 47$. Then we have $\left\lfloor
\frac{l_i}{2}-2\right\rfloor \geq \frac{21}{47} l_i$ and $\left\lfloor \frac{l_i}{3} \right\rfloor \geq \frac{15}{47}
l_i$ and the inequality $e^{m_i^2 \cdot 2 l_i} \leq \left(2^{l_i-1} \cdot
\left(l_i-1\right)^{\left\lfloor\frac{l_i}{2}\right\rfloor -2} \right)^{\left\lfloor\frac{m_i}{3^i}\right\rfloor}$
becomes

\[3^i m_i^2 \cdot 2 l_i \leq \left(\frac{15}{47}\right)^i m_i \cdot \frac{21}{47} l_i \log \left(l_i-1\right).\] 
Cancellation now gives $5 m_i \cdot  \left(\frac{47}{5}\right)^i \leq \log\left(l_i-1\right)$.

\end{proof}

\section{Non-Trivial Words}

Our object in this section is to show that if the defining sequence satisfies $l_i \geq \left(25 l_{i-1}\right)^{3 m_i}$
for all $i \in \mathbb{N}$ then the group constructed above has no free subgroups of rank $2$. Indeed, given any two
elements $g_1$, $g_2$ of the group we construct explicitly a non-trivial word $w$ in the free group of rank $2$ such
that $w\left(g_1, g_2\right)=1$.

\medskip

It follows from Proposition \ref{prop_spinal} that we can write every $g \in G$ as $g = a^k \prod_{i=1}^{s_g}
b\left(k_i\right)^{q_i}$ with $k, q_i \in \mathbb{Z}$, $k_i \in \left\{1, \dots, l_0\right\}$  and $s_g \in \mathbb{N}$.

\begin{definition}
A \emph{spine} $s = g^{-1} b^q g$ is a power of a $g$-conjugate $b$ with $g \in G$ and some $q \in
\mathbb{Z}\backslash\left\{0\right\}$. Denote by \[\xi( g ) = \min \left\{s_g: g = a^k
\prod_{i=1}^{s_g}
b\left(k_i\right)^{q_i}\right\}\] the number of spines of $g$ for any $g \in G$.
\end{definition}

\begin{remark}
The number of spines $\xi(g)$ should not be confused with the word length of $g$ if $g \in B$ as a word in the
generators of
$B$, $\left\{ b(1), \dots, b\left(l_0\right) \right\}$.
\end{remark}

\begin{lemma} \label{lemma_spineInequ}
$\xi( gh) \leq \xi( g ) + \xi( h )$ and hence $\xi\left( g^h \right) \leq \xi (g) + 2 \xi (h)$ for any $g,h \in G$.
\end{lemma}

\begin{proof}
This follows immediately from the definition.
\end{proof}

Let $g_1, g_2 \in G$. Recursively define commutators $c_{1}= \left[g_1, g_2\right]$ and $c_i =
\left[c_{i-1}, c_{i-1}^{c_{i-2}}\right]$ for $i \geq 2$ with $c_0 = g_1$. Then we get the following lemma:

\begin{lemma}
If $g_1, g_2 \in G$, then the number of spines $\xi\left( c_i \right)$ in the commutator $c_i$ defined as above is
bounded by $\xi\left( c_i \right)  \leq 5^i \left( \xi\left(
g_1 \right) + \xi\left( g_2 \right) \right)$ for all $i \geq 0$.
\end{lemma}

\begin{proof}
Using that $c_{i-1}=\left[c_{i-2}, c_{i-2}^{c_{i-3}}\right]$ gives $4 \xi\left(c_{i-2}\right) + 2
\xi\left(c_{i-3}\right) \leq \xi\left(c_{i-1}\right)$ and hence $\xi\left( c_{i-2} \right) \leq \frac14 \xi\left(
c_{i-1}\right)$. This gives $\xi\left(c_i \right) \leq 4 \xi \left(c_{i-1}\right) + 2\xi \left(c_{i-2}\right) \leq 5
\xi\left( c_{i-1} \right)$.
\end{proof}

The strategy is to observe that the number of spines of the commutators $c_i$ grows more slowly than the
number of vertices on each level. We note the position of the spines of $c_i$ and aim to move them by conjugation such
that none of the conjugated spines is at an old position. This new element will then commute with $c_i$. The following
combinatorial proposition will be needed to ensure that such a shift is possible.

\begin{prop}\label{prop_shift}
Let $N \subset \mathbb{N}$ be a finite set with $|N|=n$. Then there exists some $0 < q < n^2$ such that $N \cap N_q =
\emptyset$ where $N_q = \left\{k+q | k \in N\right\}$.
\end{prop}

\begin{proof}
Look at the set $D = \left\{k_i - k_j | k_i, k_j \in N\right\}$. This set has at most $|D| \leq \left(|N|-1\right)^2 +
1$ elements because we get the value zero $n$ times. The elements of this set are exactly the values which we cannot
choose for $q$. Hence there exists some $0 < q < n^2$ with the required property.
\end{proof}

\begin{lemma}\label{lemma_cistab}
For every $i \geq 1$ we have $c_i \in \rst_G(i)$. 
\end{lemma}

\begin{proof}
We have $c_1 \in G' \leq \rst_G(1) \lhd G$. Hence $c_1^{g_1} \in \rst_G(1)$ and so $c_2 = \left[c_1, c_1^{g_1}\right]
\in
\rst_G(1)' \leq st_G(1)' \leq \rst_G(2)$. Now assume $c_n \in \rst_G(n)$. Then $c_n^{c_{n-1}} \in \rst_G(n)$ and hence
again
$c_{n+1} = \left[c_n, c_n^{c_{n-1}}\right] \in \rst_G(n)' \leq \St_G(n)' \leq \rst_G(n+1)$.
\end{proof}

\begin{prop}\label{prop_commShape}
The commutators $c_i$ have the recursive form $c_i = \left(d_{i,1}, \dots, d_{i, m_i}\right)_i$ where each 
$d_{i,j}$ falls into one of the four cases:

\begin{enumerate}
 \item $d_{i,j}=1$,
 \item $d_{i,j}=b_i^t$ for $t \in \mathbb{Z}$,
 \item $d_{i,j} = a^q b$ with $q \neq 0 \mod l_i$ and $b \in B_i$ or
 \item $d_{i,j} = \left(d_{i+1,1+(j-1) l_i},\dots, d_{i+1,j \cdot l_i}\right)$. 
\end{enumerate}

Further, there exists some level $n$ such that all $d_{i,j}$ will have fallen into one of the first three cases.
\end{prop}

\begin{proof}
From Lemma \ref{lemma_cistab} we have $c_i=\left(g_1, \dots, g_{m_{i}}\right) \in G_{i} \times \dots \times G_{i}$ with
$g_j = a_i^{q_j} \prod_{k=0}^{u_j} b_i\left(r_{i,k}\right)^{f_{i,k}}$ and $q_j,f_{i,k} \in \mathbb{Z}, u_j \in
\mathbb{N}, r_{i,k} \in \left\{1, \dots, l_i\right\}$. If $g_i$ is an element of $B \leq G_{i+1} \times \dots \times G_{i+1}$ but not of the form $b_i^t$ for some $t \in \mathbb{Z}$ then $d_{i,j}$ was such that $d_{i,j} = \left(d_{i+1,1+(j-1) l_i},\dots, d_{i+1,j \cdot l_i}\right) \in \St_G(i+1)$. Assume that at least one $d_{i+1, h}$ is again of this form, the forth case. Then $d_{i+1,h} = a_{i+1}^{q_h} \prod_{s=1}^{y_h} b_{i+1}\left(f_{h,s}\right)^{z_{h,s}}$ with $q_h, z_{h,s} \in \mathbb{Z}$,
$y_h \in \mathbb{N}$ and $f_{h,s} \in \left\{1, \dots, l_i\right\}$. We assume that not all $f_{h,s}$ are equal to $1$
and that $q_h =0$ to eliminate cases $2$ and $3$. However, if there exists a $f_{h,s_0} \neq
1$ then $d_{i,j}$ was such that $b_{i+1}\left(f_{h,s_0}\right) = b_{i}(c)^{b_{i}(c-1)^q}$ for some $c\in \left\{1,
\dots, l_{i-1}\right\}$ and some $q \in \mathbb{Z}\backslash \{0\}$. This yields that the word lengths satisfy
$|d_{i+1,h}| < |d_{i,j}|-1$ and hence there exists a level $n$ such that all $d_{n,m}$ fall into one of the first three
cases.
\end{proof}

Write $\left\lfloor i/j\right\rfloor$ for the biggest integer $q$ such that $q \leq i/j$.

\begin{cor}\label{cor_commShape}
For every $d_{i,j}$ in $c_i$ that is of the second or third type we have that either \[d_{i-1,\left\lfloor j/l_i\right\rfloor}=a^q b
\textnormal{ or } d_{i-1,\left\lfloor j/l_i\right\rfloor}^{d_{i-2,\left\lfloor j/\left(l_{i-1} l_i\right)\right\rfloor}}= a^q b\]
with $q \neq 0 \mod l_i$, hence at least one of the two was of type $3$.
\end{cor}

\begin{proof}
If $h,k$ are of type $1,2$ or $4$ then $h,k \in B_i \times \dots \times B_i$ and hence \[[h,k] \in B_i' \times \dots
\times B_i' \leq G_i'' \times \dots \times G_i'' \leq \rst_G(i+1) \leq \St_G(i+2)\] and hence cases $2$ and $3$ are
impossible.
\end{proof}

\begin{theorem}\label{thm_freeSubgroups}
Assume that the defining sequence $\left\{ l_i\right\}$ satisfies $l_i \geq \left(25l_{i-1}\right)^{3 \prod_{j=0}^{i-1}
l_j}$. Then $G$ has no free subgroup of rank $2$.
\end{theorem}

\begin{proof}
Let $g_1, g_2 \in G$. Then we construct a non-trivial word $w_{g_1,g_2}(x,y)$ such that  we have $w_{g_1,
g_2}\left(g_1, g_2 \right)=1$. Set $s_i = 5^{i} \left(\xi\left( g_1 \right) + \xi\left( g_2
\right)\right)$, the number of
spines in the commutator $c_i$ as defined above. Find a level $k$ such that \[s_0 = \xi\left( g_1
\right) + \xi\left( g_2 \right) \leq 5^k\] and further use the fact that $2k + 1 \leq m_k$ for all $k \geq 0$. This
implies \[2 s_k^{3} \leq 2 \cdot 5^{3k} s_0^3 \leq 2 \cdot 5^{3k} 5^{3k} \leq 5^{6k+1} \leq 5^{3m_k} \leq l_k.\]
Corollary \ref{cor_commShape} yields that every non-trivially decorated vertex in $c_k$ has a rooted
decoration on the vertex immediately above it. Write $c_k = \left(d_{i_1,1}, \dots, d_{{i_r},y}\right)$ where different
$d_{i,j}$ will now in general lie in different levels $i_n$. Denote by \[D =
\left\{d_{i,j} | d_{i,j} \textnormal{ occurs in } c_k\right\}\] and assume an ascending lexicographic order with respect to the double indices $\left\{i,j\right\}$. Let $j_D =
|D|$. Pick the first element $d_{i_1, j_1}$ in $D$ and let $v$ be the $j_1$-th vertex of
 level $i_1$, the one $d_{i_1, j_1}$ acts on. By Corollary \ref{cor_commShape} we have either $c_{k-1}|_v = a_{k-1}^t b$
or $c_{k-1}^{c_{k-2}}|_v = a_{k-1}^t b$ with $t \neq 0 \mod l_{k-1}$ and $b \in B_{k-1}$. Assume that we can find $m$
such that $mt \equiv q \mod l_i$ for $q$ from Proposition \ref{prop_shift}. 
Then either \[d= \left[c_k, c_k^{c_{k-1}^m}\right] \quad\hbox{or}\quad e = \left[c_k,
c_k^{\left(c_{k-1}^{c_{k-2}}\right)^m}\right]\] is such that $d|_v = 1$ or $e|_v = 1$ and we get $\xi\left( d \right)
\leq 5 l_{k-1} \xi\left( c_k \right)$ and $\xi(e) \leq 5 l_{k-1}\xi\left( c_k \right)$. Repeat this with $D_c =
D\backslash \left\{d_{i_1,j_1}\right\}$ where we have that $j_{D_c} < j_D$ until $j_{D_c}=0$.

We now have to justify that we can find such a power $m$ for all elements in $D$. At most all $m_k$
vertices on level $k$ have non-trivial decoration. Because of the recursive case of Proposition
\ref{prop_commShape} we could have to go further down on some parts of the tree.

In the worst case we have to go to level $r \geq k$ for every $d_{i,j} \in D$. Then we have $m_r$ vertices to look at
and powers less than $l_{r-1}$, leading to at most $s = \left(5 l_{r-1}\right)^{m_r} s_k$ spines. These
spines will have to be shifted among $l_r$ vertices. We need $2s^2$ places to perform the shift where the factor $2$
occurs because every spine has a rooted element next to it and hence actually decorates two places. We need to make
sure that the smallest possible size of the biggest gap between spines is at least $2 s^2$. This size is at least
$l_r/s$ and hence we require $l_r \geq 2 s^3$ and so $l_r \geq 2 \left(\left(5 l_{r-1}\right)^{m_r} s_k \right)^3$. The
last term is less than or equal to $2 \left(\left(5 l_{r-1}\right)^{m_r} s_r
\right)^3$. Hence it is sufficient to have 
\begin{equation}\label{eq_l_m}
l_r \geq 2 s_r^3 \left(5l_{r-1}\right)^{3 m_r}.
\end{equation}
By our hypothesis on $k$ that $\xi\left(g_1\right) + \xi\left(g_2 \right) \leq 5^k \leq 5^r$ we get that $s_r = 5^r
\left( \xi\left(g_1\right) + \xi\left(g_2 \right)\right) \leq 5^{2r}$ and hence $2 s_r^3 \leq 2 \cdot 5^{6r} \leq
5^{6r+1}$. Hence the
last term in (\ref{eq_l_m}) is less than or equal to $\left(25l_{r-1}\right)^{3m_r}$ because we have $2r + 1 \leq m_r$.
This
yields that the
procedure described above will result in a non-trivial word in $\left<g_1, g_2 \right>$ and hence $G$ cannot contain a
non-abelian free subgroup.
\end{proof}

This immediately implies that $G$ is cannot be large in this case. However, we can prove for any coprime sequence $l_i$
with $l_i \geq 7$ that $G$ is not large:

\begin{theorem}\label{thm_large}
The group $G$ is not large.
\end{theorem}

\begin{proof}
Assume for a contradiction that $G$ is large. Then there exists a finite index subgroup $H$ that maps onto the
non-abelian free group of rank $2$, hence also onto the alternating group $A_5$. Denote the kernel of the canonical
map $H \rightarrow A_5$ by
$N \leq H$. Then
$N = \bigcap_{g \in G} N^g$ is a proper normal subgroup of $G$. The quotient $G/N$ is soluble by Proposition
\ref{prop_quotientSoluble} and hence cannot have a section isomorphic to $A_5$.
\end{proof}

\end{document}